\documentclass{amsart}

\usepackage[utf8]{inputenc}
\usepackage{amssymb}
\usepackage{amsthm}
\usepackage{amsmath}
\usepackage{tikz}

\usepackage{enumitem}
\usepackage{nicefrac}

\usepackage{tikz-cd}

\usepackage{tikz-cd}

\usepackage{tikz}
\usetikzlibrary{datavisualization}
\usetikzlibrary{datavisualization.formats.functions}
\usetikzlibrary{calc}
\tikzset{stretch/.initial=1}
\newcommand\drawloop[4][]%
   {\draw[shorten <=0pt, shorten >=0pt,#1]
      ($(#2)!\pgfkeysvalueof{/tikz/stretch}!(#2.#3)$)
      let \p1=($(#2.center)!\pgfkeysvalueof{/tikz/stretch}!(#2.north)-(#2)$),
          \n1= {veclen(\x1,\y1)*sin(0.5*(#4-#3))/sin(0.5*(180-#4+#3))}
      in arc [start angle={#3-90}, end angle={#4+90}, radius=\n1]%
   }
\usetikzlibrary{arrows, decorations.pathreplacing, positioning,cd}
\usepackage{pgffor}

\usepackage{mathrsfs}

\usetikzlibrary{calc}

\makeatletter
\pgfarrowsdeclare{X}{X}
{
  \pgfutil@tempdima=0.3pt%
  \advance\pgfutil@tempdima by.25\pgflinewidth%
  \pgfutil@tempdimb=5.5\pgfutil@tempdima\advance\pgfutil@tempdimb by.5\pgflinewidth%
  \pgfarrowsleftextend{+-\pgfutil@tempdimb}
  \pgfarrowsrightextend{0pt}
}
{
  \pgfutil@tempdima=0.3pt%
  \advance\pgfutil@tempdima by.25\pgflinewidth%
  \pgfsetdash{}{+0pt}
  \pgfsetroundcap
  \pgfsetmiterjoin
  \pgfpathmoveto{\pgfqpoint{-5.5\pgfutil@tempdima}{-6\pgfutil@tempdima}}
  \pgfpathlineto{\pgfqpoint{5.5\pgfutil@tempdima}{6\pgfutil@tempdima}}
  \pgfpathmoveto{\pgfqpoint{-5.5\pgfutil@tempdima}{6\pgfutil@tempdima}}
  \pgfpathlineto{\pgfqpoint{5.5\pgfutil@tempdima}{-6\pgfutil@tempdima}}
  \pgfusepathqstroke
}
\makeatother

\newtheorem{theorem}{Theorem}[section]
\newtheorem{lemma}[theorem]{Lemma}
\newtheorem{proposition}[theorem]{Proposition}
\newtheorem{corollary}[theorem]{Corollary}

\theoremstyle{definition}
\newtheorem{definition}[theorem]{Definition}
\newtheorem{example}[theorem]{Example}
\theoremstyle{remark}
\newtheorem{remark}[theorem]{Remark}

\usepackage{relsize}
\usepackage{exscale}

\usepackage{mathtools}
\usepackage{mathrsfs}

\usepackage{framed}

\usepackage{datetime}

\usepackage{hyperref} 

\renewcommand{\epsilon}{\varepsilon}

\DeclareMathOperator{\Orb}{Orb}

\title{Expansivity and unique shadowing}
\author[Good, Mac\'{\i}as, Meddaugh, Mitchell and Thomas]{Chris Good, Sergio Mac\'{\i}as, Jonathan Meddaugh, Joel Mitchell and Joe Thomas}
\date{February 2020}

\begin{document}

\hypersetup{pageanchor=false} 

\begin{abstract}
Let $f\colon X\to X$ be a continuous function on a compact metric space. We show that shadowing is equivalent to backwards shadowing and two-sided shadowing when the map $f$ is onto. Using this we go on to show that, for expansive surjective maps the properties shadowing, two-sided shadowing, s-limit shadowing and two-sided s-limit shadowing are equivalent. We show that $f$ is positively expansive and has shadowing if and only if it has unique shadowing (i.e.\ each pseudo-orbit is shadowed by a unique point), extending a result implicit in Walter's proof that positively expansive maps with shadowing are topologically stable. We use the aforementioned result on two-sided shadowing to find an equivalent characterisation of shadowing and expansivity and extend these results to the notion of $n$-expansivity due to Morales.
\end{abstract}

\maketitle


\hypersetup{pageanchor=true} 





\section{Introduction}
Let $f\colon X\to X$ be a continuous function on a compact metric space $X$. A $\delta$\emph{-pseudo-orbit} is a sequence $(x_i)_{i\in \mathbb{N}_0}$ such that $d(f(x_{i}),x_{i+1}) <\delta$. Pseudo-orbits are of importance when calculating an orbit numerically, as rounding errors mean a computed orbit will be a pseudo-orbit. The sequence $(y_i)$ from $X$ is said to $\epsilon$-shadow the sequence $(x_i)$ provided $d(y_i,x_i)<\epsilon$ for all $i$. We then say that the system has \emph{shadowing}, or \emph{the pseudo-orbit tracing property}, if pseudo-orbits are shadowed by true orbits (see below for precise definitions). Motivating this paper is Walters \cite{WaltersP} result that if $h$ is an expansive homeomorphism with shadowing, then for every $\epsilon>0$ there is a $\delta>0$ such that every $\delta$-pseudo-orbit is $\epsilon$-shadowed by a unique point from $X$. We show that the converse is true; a system is shadowing and expansive if and only if it has unique shadowing. We go on to obtain results of a similar flavour using the notion of $n$-expansivity due to Morales \cite{morales}.

Shadowing is important when modelling a system numerically (for example see \cite{Corless, Pearson}). However, it is also important theoretically. For example, Bowen \cite{bowen-markov-partitions} used shadowing implicitly as a key step in his proof that the nonwandering set of an Axiom A diffeomorphism is a factor of a shift of finite type. Since then it has been studied extensively, in the setting of numerical analysis \cite{Corless,CorlessPilyugin,Pearson}, as a key factor in stability theory \cite{Pilyugin, robinson-stability,WaltersP}, in understanding the structure of $\omega$-limit sets and Julia sets, 
\cite{barwell-davies-good, BarwellGoodOprochaRaines, barwell-raines-meddaugh, barwell-raines, Bowen, GoodMeddaughMitchell,
MeddaughRaines}, and as a property in and of itself \cite{Coven, fernandez-good, GoodMeddaugh2018, LeeSakai, Nusse, Pennings, Pilyugin,Sakai2003}.

Many other notions of shadowing have been studied including, for example, ergodic, thick and Ramsey shadowing \cite{brian, brian-oprocha, bmr, Dastjerdi, Fakhari, oprocha}, limit shadowing \cite{BarwellGoodOprocha, good-oprocha-puljiz, Pilugin2007}, s-limit shadowing \cite{BarwellGoodOprocha, good-oprocha-puljiz, LeeSakai}, orbital shadowing \cite{GoodMeddaugh2016, Mitchell, PiluginRodSakai2002, Pilugin2007}, and inverse shadowing \cite{CorlessPilyugin, GoodMitchellThomas2, Lee}. In this paper we focus on shadowing, s-limit shadowing, h-shadowing and limit shadowing.

In Section \ref{section-shadowing}, we observe (Theorem \ref{theorem-shadowing}) that if $f$ is surjective then it has shadowing if and only if for any $\epsilon >0$ there exists $\delta >0$ such that every backwards $\delta$-pseudo orbit is $\epsilon$-shadowed by some backwards orbit of a point: thus shadowing is equivalent to \emph{backwards shadowing}. We additionally show that it is equivalent to \emph{two-sided shadowing} (i.e.\ two-sided pseudo-orbits are shadowed by a two-sided trajectory of a point). We then strengthen a result in \cite{BarwellGoodOprochaRaines} (Corollary \ref{CorollaryEquivShad_and_sLim}), by demonstrating that for expansive maps, the properties shadowing, two-sided shadowing, s-limit shadowing and two-sided s-limit shadowing are equivalent. In Section \ref{section-unique}, we turn our attention to the notion of $n$-expansivity due to Morales \cite{morales}. We show (Theorem \ref{thmnShadIFF}) that pseudo-orbits are shadowed by at most $n$ points if and only if $f$ has shadowing and is $n$-expansive. We then construct an example of a positively $n$-expansive system with shadowing which is not positively $(n-1)$-expansive. We close by examining the consequences of uniqueness in three other shadowing properties, namely s-limit shadowing, limit shadowing and h-shadowing.
\section{Preliminaries}
This section serves to outline the preliminary background definitions and notions for the remainder of this paper and are standard across the literature. Throughout, we will assume that a discrete \emph{dynamical system} is a pair $(X,f)$ consisting of a compact metric space $X$ and a continuous map $f\colon X \to X$. Note that we do not assume, in general, that the map $f$ is onto. However, since surjective dynamical systems are usually the more interesting from a dynamics viewpoint, we ensure that every example we construct in this paper is surjective (unless it is the property of surjectivity itself which is under examination). We say that the \emph{orbit} of $x$ under $f$ is the set of points $\{x, f(x), f^2(x), \ldots\}$; we denote this set by $\Orb_f(x)$. A (finite or infinite) sequence $(x_i ) _{0 \le i\le n}$ for some $n\in \mathbb{N}\cup\{\infty\}$ is said to be a \emph{$\delta$-pseudo-orbit} for some $\delta>0$ if $d(f(x_i),x_{i+1})<\delta$ for each $i \le n$. The infinite sequence $(x_i ) _{i \in \mathbb{N}_0}$ is an \emph{asymptotic pseudo-orbit} provided that $\lim_{i \rightarrow \infty} d(f^i(x_i), x_{i+1}) =0$ and we say that $( x_i ) _{i \in \mathbb{N}_0}$ is an \emph{asymptotic $\delta$-pseudo-orbit} if it is both a $\delta$-pseudo-orbit and an asymptotic pseudo-orbit.
The point $z \in X$ is said to \emph{$\epsilon$-shadow} $(x_i ) _{0 \le i\le n}$ for some $\epsilon>0$ if $d(x_i,f^i(z))<\epsilon$ for each $i \le n$. It \emph{asymptotically shadows} the sequence $(x_i)_{i\in\mathbb{N}_0}$ if $\lim_{i \rightarrow \infty} d(x_i, f^i(z)) =0$ and \emph{asymptotically $\epsilon$-shadows} the sequence if it both $\epsilon$-shadows and asymptotically shadows it.

The classical notion of shadowing states that $(X, f)$ has \emph{shadowing} provided for any $\epsilon >0$ there exists $\delta >0$ such that every (infinite) $\delta$-pseudo-orbit is $\epsilon$-shadowed. The system has \emph{limit shadowing},  a property first introduced in \cite{EirolarNevanlinnaPilyugin} with reference to hyperbolic sets, if every asymptotic pseudo-orbit is asymptotically shadowed. The notion of limit shadowing was extended in \cite{LeeSakai} to a property the authors called s-limit shadowing to accommodate the fact that many systems exhibit limit shadowing but not shadowing \cite{Kulczycki,Pilyugin}. The system $(X,f)$ has \emph{s-limit shadowing} if, in addition to having shadowing\footnote{We note that postulating shadowing as part of the definition of s-limit shadowing is actually unnecessary when the phase space is compact (see by \cite[Theorem 11.0.1]{GoodMitchellThomas}).}, for any $\epsilon>0$ there exists $\delta>0$ such that for any asymptotic $\delta$-pseudo orbit $( x_i )_{i \in \mathbb{N}_0}$ there exists $z \in X$ which asymptotically $\epsilon$-shadows $( x_i )_{i\in\mathbb{N}_0}$. Finally, the system $(X,f)$ has \emph{h-shadowing}, or \emph{shadowing with exact hit}, if for any $\epsilon>0$ there exists $\delta>0$ such that for any finite $\delta$-pseudo orbit $(x_0, x_1, \ldots x_m)$ there exists $z \in X$ which $\epsilon$-shadows it and for which $f^m(z)=x_m$.

We remark that h-shadowing was introduced in \cite{BarwellGoodOprochaRaines} and was motivated by the fact that an important class of shift systems, called shifts of finite type, which are fundamental in the study of shadowing (see \cite{GoodMeddaugh2018}) exhibit this stronger form of shadowing and that it coincide with the usual form for shift systems but is distinct in general (see \cite[Example 6.4]{BarwellGoodOprocha}). Moreover, it is known from results in \cite{BarwellGoodOprocha} that $h$-shadowing implies s-limit shadowing which further implies implies limit shadowing.

\section{Two-sided shadowing}\label{section-shadowing}

We start with the following simple observation, which we nevertheless believe to be new for functions in general. The classical notion of shadowing states that $(X, f)$ has shadowing provided for any $\epsilon >0$ there exists $\delta >0$ such that every $\delta$-pseudo-orbit is $\epsilon$-shadowed. It is a standard result in the theory of shadowing \cite{Pilyugin} that a compact dynamical system $(X,f)$ has shadowing if and only if for any $\epsilon>0$ there is a $\delta>0$ such that every finite $\delta$-pseudo orbit $(x_0,\ldots,x_n)$ is $\epsilon$-shadowed by some $x\in X$. It is shown here that in a compact space, one obtains an equivalent notion of shadowing in terms of backwards and two-sided (pseudo-)orbits.

\begin{definition}\label{def: twosided and backward pseudo orbits}
Suppose that $(X,f)$ is a dynamical system.
\begin{enumerate}
\item A \emph{backwards orbit} of the point $x\in X$ is a sequence $(x_{i})_{i\leq0}\subseteq X$ for which $f(x_{i})=x_{i+1}$ for all $i\leq -1$ and $x_0=x$.
\item A \emph{two-sided orbit} of the point $x\in X$ is a sequence 
$(x_i)_{i\in\mathbb Z}\subseteq X$ for which $f(x_i)=x_{i+1}$ for all $i\in\mathbb Z$ and $x_0=x$.
\item The sequence $(x_{i})_{i\leq0}\subseteq X$ is a \emph{backwards $\delta$-pseudo-orbit} if $d(f(x_{i}),x_{i+1})<\delta$ for each $i\leq-1$.
\item The sequence $(x_{i})_{i\in\mathbb Z}\subseteq X$ is a \emph{two-sided $\delta$-pseudo-orbit} if $d(f(x_{i}),x_{i+1})<\delta$ for each $i\in\mathbb Z$.
\item $(X,f)$ is said to have the \textit{backwards shadowing property} if for any $\varepsilon>0$, there exists $\delta>0$ for which every backwards $\delta$-pseudo-orbit in $X$ is $\varepsilon$-shadowed by some backwards orbit of a point in $X$.
\item $(X,f)$ is said to have the \textit{two-sided shadowing property} if for any $\varepsilon>0$, there exists $\delta>0$ for which every two-sided $\delta$-pseudo-orbit in $X$ is $\varepsilon$-shadowed by some two-sided orbit of a point in $X$.
\end{enumerate}
\end{definition}

Obviously if $f$ is not a homeomorphism, backwards and two-sided orbits need not be unique. 

\begin{theorem}\label{theorem-shadowing}
Let $(X,f)$ be a dynamical system with $X$ compact. Then, of the following, (1) implies (2) which implies (3). Furthermore, if $f$ is onto then (3) implies (1).
\begin{enumerate}
\item $f$ has shadowing;
\item $f$ has two-sided shadowing;
\item $f$ has backwards shadowing.
\end{enumerate}   
\end{theorem}

\begin{proof} $(1) \implies (2)$: Suppose that $(X,f)$ has shadowing. Let $\epsilon>0$ and choose $\delta>0$ such that every $\delta$-pseudo-orbit is $\epsilon/2$-shadowed. Suppose that $(x_{n})_{n\in\mathbb Z}$ is a two-sided $\delta$-pseudo-orbit. 
For each $n>0$, let $y_{-n}$ be a point which $\epsilon/2$-shadows the $\delta$-pseudo-orbit $(x_{-n}, x_{-n+1}, x_{-n+2},\ldots)$. There exists a point $z_0\in X$ and an infinite subset $N_0$ of $\mathbb{N}_0$ such that $f^n(y_{-n})\to z_0$ as $n\to \infty$ and $n\in N_0$. Clearly the forward orbit of $z_0$ $\epsilon$-shadows $(x_0, x_1, x_2,\ldots)$. Given $z_{-k}$ and an infinite subset $N_k$ of $\mathbb{N}_0$ such that $f^{n-k}(y_{-n}) \to z_{-k}$ as $n\to \infty$ and $n\in N_k\cap\{k+1,k+2,\ldots\}$, we can find a point $z_{-k-1}$ and an infinite $N_{k+1}\subseteq N_k$ such that $f^{n-k-1}(y_{-n}) \to z_{-k-1}$ as $n\to \infty$ and $n\in N_{k+1}\cap\{k+2,k+3,\ldots\}$. Note that $d( x_{-k},z_{-k})<\epsilon$ and that, by continuity, $f(z_{-k-1})=z_{-k}$ for all $k\ge0$. Hence $z_0$ has a two-sided orbit that $\epsilon$-shadows $(x_{n})_{n\in\mathbb{Z}}$.

It is clear that (2) implies (3). Finally (3) implies (1) because, given that every point has a pre-image, (3) implies that finite pseudo-orbits are shadowed, which is equivalent to shadowing in compact metric spaces.
\end{proof}

 One can also extend the notion of s-limit shadowing to the two-sided and backward varieties. For this, one requires the notions of two-sided asymptotic pseudo-orbits and backward asymptotic pseudo-orbits. These are defined analogously to the normal (forward) asymptotic pseudo-orbits but in the spirit of Definition \ref{def: twosided and backward pseudo orbits}. 
 
 \begin{definition}
 \label{def: twosided+backward slim shadowing}
 Suppose that $(X,f)$ is a dynamical system.
\begin{enumerate}
\item A \emph{backwards asymptotic $\delta$-pseudo-orbit} is a backwards $\delta$-pseudo orbit $(x_i)_{i\leq 0}$ which in addition satisfies $d(f(x_i),x_{i+1})\to 0$ as $i\to\infty$.
\item A \emph{two-sided asymptotic $\delta$-pseudo-orbit} is a two-sided $\delta$-pseudo orbit $(x_i)_{i\in\mathbb{Z}}$ which in addition satisfies $d(f(x_i),x_{i+1})\to 0$ as $i\to\pm\infty$.
\item $(X,f)$ is said to have the \textit{backwards s-limit shadowing property} if it has the backwards shadowing property and for any $\varepsilon>0$ there exists a $\delta>0$ such that every backwards asymptotic $\delta$-pseudo-orbit in $X$ is asymptotically $\varepsilon$-shadowed by some backwards orbit in $X$.
\item $(X,f)$ is said to have the \textit{two-sided s-limit shadowing property} if it has the two-sided shadowing property and for any $\varepsilon>0$ there exists a $\delta>0$ such that every two-sided asymptotic $\delta$-pseudo-orbit in $X$ is asymptotically $\varepsilon$-shadowed by some two-sided orbit in $X$.
\end{enumerate}
\end{definition}
We then obtain a connection between the different varieties of s-limit shadowing.

\begin{proposition}\label{prop_Twosided_slim_Implies_Other_slim}
If $f$ has two-sided s-limit shadowing then it has backward s-limit shadowing. If, in addition, $f$ is a surjection then $f$ has s-limit shadowing.
\end{proposition}

\begin{proof}
Let $\epsilon>0$ be given and let $\delta>0$ correspond to this for two-sided s-limit shadowing.

Let $(x_i)_{i\leq 0}$ be a backward asymptotic $\delta$-pseudo-orbit. Extend this into a two-sided asymptotic $\delta$-pseudo-orbit by letting $x_i=f^i(x_0)$ for all $i >0$. By two-sided s-limit shadowing there exists $z \in X$ which asymptotically $\epsilon$-shadows $(x_i)_{i \in \mathbb{Z}}$. In particular, $z$ backwards asymptotically $\epsilon$-shadows $(x_i)_{i\leq 0}$. This part of the result now follows by the fact that two-sided shadowing implies backward shadowing (see the proof of Theorem \ref{theorem-shadowing}).

Now let $(x_i)_{i \geq 0}$ be an asymptotic $\delta$-pseudo-orbit and suppose $f$ is onto: for each $i <0$ let $x_i$ be such that $f(x_i)=x_{i+1}$. By two-sided s-limit shadowing there exists $z \in X$ which asymptotically $\epsilon$-shadows $(x_i)_{i \in \mathbb{Z}}$. In particular, $z$ asymptotically $\epsilon$-shadows $(x_i)_{i\leq 0}$. It remains to note that $f$ has shadowing by Theorem \ref{theorem-shadowing}.
\end{proof}
The following example shows the necessity of surjectivity in the previous result. Indeed, one can exhibit a non-surjective system with two-sided s-limit shadowing but not s-limit shadowing. 

\begin{example}\label{ExampleNotOnto}
Let $X= \{ -1, -\nicefrac{1}{2}, 0, \nicefrac{1}{2^n} \mid n \in \mathbb{N}_0 \}$ with the induced metric from the real line. Let
\[f(x)=\left\{\begin{array}{lll}
x+\nicefrac{1}{2}& \text{if} & x \in \{-1,\nicefrac{1}{2}\},
\\
x & \text{if} & x \in \{0, 1\},
\\
\nicefrac{1}{2^{n-1}} & \text{if} & x= \nicefrac{1}{2^n} \text{ for } n\geq 1.
\end{array}\right.\]
For any $\delta$, we can construct a $\delta$-pseudo-orbit starting from $-1$ and ending with a sequence of $1$s. This cannot be, for example, $\nicefrac{1}{3}$-shadowed, thus the system does not have s-limit shadowing. However, for $\delta < \nicefrac{1}{3}$ every backward asymptotic $\delta$-pseudo orbit lies in $[0,1]$. Similarly, every two-sided asymptotic $\delta$-pseudo-orbit lies in $[0,1]$, and it is clear that the subsystem $X \cap [0,1]$ has backward and two-sided s-limit shadowing.
\end{example}
Recall that a system $(X,f)$ is \textit{c-expansive} if there exists some $\eta>0$ such that for any $x,y\in X$ and two-sided orbits $(x_i)_{i\in\mathbb{Z}}$ and $(y_i)_{i\in\mathbb{Z}}$ in $X$ with $x_0=x$, $y_0=y$ and $d(x_i,y_i)<\eta$ for all $i\in\mathbb{Z}$ one has $x=y$. It is often seen that systems with expansivity properties guarantee that certain characteristic properties of shadowing varieties hold. For example, in \cite{BarwellGoodOprochaRaines} the first named author \emph{et al} show that an expansive map has shadowing if and only if it has s-limit shadowing. The next result extends this further by providing an equivalence between s-limit shadowing and two-sided s-limit shadowing.

\begin{theorem}\label{thmcExpansiveSHadIFFsLimSHad}
Let $(X,f)$ be a dynamical system. If $f$ is c-expansive then $f$ has two-sided shadowing if and only if $f$ has two-sided s-limit shadowing.
\end{theorem}
 \begin{proof}
 If $f$ has two-sided s-limit shadowing then it has two-sided shadowing by definition. Therefore, suppose that $f$ has two-sided shadowing. Let $\eta>0$ be the c-expansivity constant for $f$ and take $\epsilon>0$ with $\epsilon < \nicefrac{\eta}{2}$: let $\delta>0$ correspond to this $\epsilon$ in the definition of two-sided shadowing (without loss of generality we assume $\delta < \nicefrac{\epsilon}{2}$). Let $( x_i ) _{i \in \mathbb{Z}}$ be a two-sided asymptotic $\delta$-pseudo-orbit. By two-sided shadowing, there exists a full orbit $( z_i ) _{ i \in \mathbb{Z}}$ such that $d(x_i, z_i) < \epsilon$ for all $i \in \mathbb{Z}$. The proof of Theorem 3.7 in \cite{BarwellGoodOprochaRaines} shows that under these conditions, $d(z_i, x_i) \to 0$ as $i \to \infty$ and thus it suffices to show that $d(z_i, x_i) \to 0$ as $i \to -\infty$. This can be done by a similar argument to that of \cite{BarwellGoodOprochaRaines}.
 
 Suppose that $d(z_i, x_i)$ does not converge to $0$ as $i \to -\infty$. Then by compactness of $X$ there exists $a_0,b_0 \in X$ and an infinite set of negative integers, $N_0$, such that
 \begin{enumerate}[label=\roman*).]
\item $\lim_{i \to -\infty, i \in N_0} x_i = a_0$,
\item $\lim_{i \to -\infty, i \in N_0} z_i = b_0$;
\item $d(a_0, b_0)=r >0$.
\end{enumerate} 
Note that by the fact that $( z_i ) _{i \in \mathbb{Z}}$ $\epsilon$-shadows $( x_i ) _{i \in \mathbb{Z}}$, it follows that $r=d(a_0, b_0) \leq \epsilon$. 
 By continuity, for any $k \in \mathbb{N}$, $\lim_{i \to -\infty, i \in N_0} z_{i+k} = f^k(b_0)=:b_k$. Furthermore, since $( x_i )_{i \in \mathbb{Z}}$ is a two-sided asymptotic pseudo-orbit it is in particular a backward asymptotic pseudo-orbit when restricted to $i\leq 0$. Thus, by continuity for any $k \in \mathbb{N}$, $\lim_{i \to -\infty, i \in N_0} x_{i+k} = f^k(a_0)=:a_k$. By shadowing, $d(a_k, b_k) \leq \epsilon$ for all $k \in \mathbb{N}$.
 
 By the compactness of $X$, there exist points $a_{-1}$ and $b_{-1}$ and an infinite subset $N_{-1} \subseteq N_0$ such that \begin{enumerate}[label=\roman*).]
\item $\lim_{i \to -\infty, i \in N_1} x_{i-1} = a_{-1}$,
\item $\lim_{i \to -\infty, i \in N_1} z_{i-1} = b_{-1}$.
\end{enumerate} 
 By the continuity of $f$, combined with the fact that $( x_i ) _{i \in \mathbb{Z}}$ is a backward asymptotic pseudo-orbit, we have $f(a_{-1})=a_0$ and $f(b_{-1})=b_0$. Notice that, once again, by shadowing $d(a_{-1}, b_{-1}) \leq \epsilon$. Continuing in this manner we can obtain two sequences of points, $a_0, a_{-1}, a_{-2} \ldots$, and $b_0, b_{-1}, b_{-2}\ldots$ as well as a sequence of subsets $N_0 \supseteq N_{-1} \supseteq N_{-2} \ldots$, such that for any $k \in \mathbb{N}$
  \begin{enumerate}[label=\roman*).]
  \item $i-k \leq 0$ for all $i \in N_{-k}$,
\item $\lim_{i \to -\infty, i \in N_k} x_{i-k} = a_{-k}$ and $f(a_{-k})=a_{-k+1}$,
\item $\lim_{i \to -\infty, i \in N_k} z_{i-k} = b_{-k}$ and $f(b_{-k})=b_{-k+1}$.
\end{enumerate} 
Therefore we have full orbits $( a_i ) _{i \in \mathbb{Z}}$ and $( b_i ) _{i \in \mathbb{Z}}$ for which $d(a_i, b_i) \leq \epsilon < \nicefrac{\eta}{2}$ (once again, using the fact that $( z_i ) _{i \in \mathbb{Z}}$ $\epsilon$-shadows $( x_i ) _{i \in \mathbb{Z}}$). But this is a contradiction; c-expansivity there exists $k \in \mathbb{Z}$ such that $d(a_k, b_k) \geq \eta$. Thus our initial assumption was false: we have that $d(z_i, x_i) \to 0$ as $i \to -\infty$.
 \end{proof}
 
The following is then immediate.

\begin{corollary}\label{CorollaryEquivShad_and_sLim}
 Let $(X,f)$ be a dynamical system. If $f$ is an expansive surjection then the following are equivalent:
 \begin{enumerate}
     \item $f$ has shadowing;
     \item $f$ has two-sided shadowing;
     \item $f$ has s-limit shadowing;
     \item $f$ has two-sided s-limit shadowing.
 \end{enumerate}
\end{corollary}

We note that the second property in the definition of two-sided s-limit shadowing (see Definition \ref{def: twosided+backward slim shadowing}(4)), namely that for every $\varepsilon>0$, there exists $\delta>0$ such that each asymptotic $\delta$-pseudo-orbit in $X$ is $\varepsilon$-shadowed by a two-sided orbit in $X$, has been previously studied in \cite{CarvalhoCordeiro2019}. The authors of that work coined this property as the \emph{$L$-shadowing property} and studied it in the context of dynamical systems whose mapping is a homeomorphism. We next show that under surjectivity, the $L$-shadowing property is sufficient to show two-sided shadowing. In other words, when the mapping is surjective, two-sided s-limit shadowing reduces simply to $L$-shadowing. This result is similar to that of the first, fourth and fifth named authors in \cite{GoodMitchellThomas} where it is shown that s-limit shadowing is equivalent to the second property in the definition of s-limit shadowing when the phase space is compact metric.

\begin{proposition}
 When $(X,f)$ is a surjective dynamical system, then $(X,f)$ has two-sided s-limit shadowing if and only if it has $L$-shadowing.
\end{proposition}

\begin{proof}
By the proof of Proposition \ref{prop_Twosided_slim_Implies_Other_slim}, as $f$ is onto, $(X,f)$ satisfies the first condition in s-limit shadowing (i.e.\ for any $\epsilon>0$ there exists $\delta>0$ such that for any asymptotic $\delta$-pseudo orbit $( x_i )_{i\in \mathbb{N}}$ there exists a point $z \in X$ which asymptotically $\epsilon$-shadows $(x_i )_{i \in \mathbb{N}}$). Therefore by the aforementioned result in \cite{GoodMitchellThomas}, $(X,f)$ has s-limit shadowing and, in particular, shadowing. The result now follows by applying Theorem \ref{theorem-shadowing}.
\end{proof}

\section{Unique Shadowing} \label{section-unique}

In his study of shadowing and stability, Walters \cite{WaltersP} proves that if $h$ is an expansive homeomorphism with shadowing, then for every $\epsilon>0$ there is a $\delta>0$ such that every $\delta$-pseudo-orbit is $\epsilon$-shadowed by a unique point from $X$. It turns out that the converse is true; a system is shadowing and expansive if and only if it has unique shadowing. By using a natural generalisation as seen in the work of Morales \cite{morales} of the notions of expansivity and positive expansivity, one can obtain results of a similar flavour which we exhibit in this section.

\begin{definition} Let $(X,f)$ be a dynamical system.
\begin{enumerate}
\item $(X,f)$ is said to be \textit{positively $n$-expansive}, for $n\in \mathbb{N}$, if there exists $r>0$ such that for any $x \in X$, the set \[\Gamma_+(x,r)=\{y \in X \mid \forall k \in \mathbb{N}_0 \, d(f^k(x),f^k(y)) < r\}, \] contains at most $n$ points.
\item $(X,f)$ is said to be \textit{$n$-expansive}, for $n\in \mathbb{N}$, if there exists $r>0$ such that for any $x_0$ and any two-sided orbit $(x_n)_{n\in\mathbb Z}$ of $x_0$ the set of $y_0$ such that $y_0$ has a two-sided orbit $(y_n)_{n\in\mathbb Z}$ with $d(x_i, y_i) <r$ for all $i \in \mathbb{Z}$ contains at most $n$ points.
\end{enumerate}
\end{definition}

Hence, a system is (positively) $1$-expansive precisely when it is (positively) expansive. 

\begin{definition}
A dynamical system $(X,f)$ is said to have \textit{(two-sided) $n$-shadowing} if there exists $\eta>0$ such that for any $\epsilon>0$ with $\epsilon<\eta$ there exists $\delta>0$ such that given any (two-sided) $\delta$-pseudo-orbit there exists at least one point and at most $n$ points which $\epsilon$-shadow it.
\end{definition}
We refer to the property of $1$-shadowing as \textit{unique shadowing}. We first demonstrate a basic characterisation of these shadowing properties using the expansivity notions introduced above.

\begin{theorem}\label{thmnShadIFF}
Let $X$ be a metric space. For any $n\in \mathbb{N}$, a dynamical system $(X,f)$ 
\begin{enumerate}
\item has $n$-shadowing if and only if it has shadowing and is positively $n$-expansive.
\item has two-sided $n$-shadowing if and only if it has two-sided shadowing and is $n$-expansive.
\end{enumerate}
\end{theorem}

\begin{proof}
Clearly if $(X,f)$ has $n$-shadowing then it has shadowing. Suppose for a contradiction that it is not positively $n$-expansive. Let $\eta>0$ be as in the definition of $n$-shadowing and suppose that $\epsilon>0$ is such that $\epsilon<\frac{\eta}{2}$. Then there exists $\delta>0$ ($\delta<\epsilon$) such that every $\delta$-pseudo-orbit is $\epsilon$-shadowed and by at most $n$ points. Let $x_0$ be a point such that $\Gamma_+(x_0,\epsilon)$ contains $n+1$ distinct points $x_0,x_1,\cdots,x_n$. Then $d(f^k(x_0),f^k(x_j))<\epsilon$ for all $k \geq 0$ and all $0\leq j \leq n$. Thus since $\{f^k(x_0)\}_{k \in \mathbb{N}_0}$ is a $\delta$-pseudo-orbit, and is $\epsilon$-shadowed by every such $x_j$, one obtains a contradiction to $n$-shadowing.

Now suppose $(X,f)$ has shadowing and is positively $n$-expansive. Let $r>0$ be a constant of the positive $n$-expansivity. We claim that $(X,f)$ has $n$-shadowing with $\eta=\frac{r}{2}$. Pick $\epsilon<\frac{r}{2}$ and let $\delta>0$ correspond to $\epsilon>0$ in the definition of shadowing. Suppose there exists a $\delta$-pseudo-orbit $(y_i)_{i \in \mathbb{N}_0}$ which is $\epsilon$-shadowed by $n+1$ distinct points $x_0, \ldots, x_n \in X$. Then by the triangle inequality, for all $n \in \mathbb{N}_0$ and any $i, j \in \{0,\ldots, n\}$, $d(f^n(x_i),f^n(x_j))< 2\epsilon < r$, a contradiction.

The proof of (2) can be argued similarly.
\end{proof}

\begin{corollary} If $(X,f)$ has $n$-shadowing then it has two-sided $n$-shadowing.\end{corollary}
\begin{proof}
This follows immediately by combining Theorems \ref{theorem-shadowing} and \ref{thmnShadIFF}.
\end{proof}

The converse of this is not true in general. Indeed, on infinite spaces there are no positively expansive homeomorphisms \cite{coven-keane} but there are expansive ones; the full shift on two symbols is such an example.

\begin{remark}
Since there are no positively expansive maps of the interval, no interval map has unique shadowing. 
\end{remark}

\begin{remark} \label{expansive-implies-finite-to-one}
Note that a positively $n$-expansive map on a compact metric is finite-to-one; if $f^{-1}(x)$ is infinite, then it has a limit point $z$ so that for any $r>0$, $\Gamma_+(z,r)$ is infinite. 
\end{remark}
One can also investigate how these versions of shadowing and expansivity interact with the $h$-shadowing property. Recall that a system $(X,f)$ has $h$-shadowing if for every $\epsilon>0$, there is a $\delta>0$ such that every finite $\delta$-pseudo orbit $(x_0,\dots,x_n)$ is $\epsilon$-shadowed by a point $z$ such that $f^n(z)=x_n$. It is known that $h$-shadowing implies s-limit shadowing which in turn implies limit shadowing \cite{BarwellGoodOprocha}. In \cite{BarwellGoodOprocha} it is also shown that a positively expansive map has shadowing if and only if it has h-shadowing. Carvalho and Cordiero \cite{CarvalhoCordeiro} prove that an $n$-expansive homeomorphism with shadowing has limit shadowing. Using these results, together with Theorem \ref{thmcExpansiveSHadIFFsLimSHad}, the following is almost immediate. The proof of (3) follows directly from the proof of Theorem C in \cite{CarvalhoCordeiro} given Remark \ref{expansive-implies-finite-to-one} above.
 
 \begin{corollary}\label{corollaryUniqueshad_hshad_limitshad_s-lim_shad}
 Let $f \colon X\to X$ be a continuous map on the compact metric space $X$. 
\begin{enumerate}
\item $f$ has unique shadowing if and only if it has h-shadowing and is positively expansive.
\item $f$ has two-sided unique shadowing if and only if it has two-sided s-limit shadowing and is expansive.
\item If $f$ is a positively $n$-expansive surjection and has shadowing, then it has limit shadowing. 
\end{enumerate}
 \end{corollary}

One may question how distinct the different notions of $n$ (positive) expansivity are for different values of $n$. This has been investigated previously in the context of homeomorphism systems. For example, Li and Zhang \cite{li-zhang-generalized} construct homeomorphisms that are (positively) $n$-expansive but not (positively) $(n-1)$-expansive for any $n\geq 2$. In \cite{CarvalhoCordeiro}, Carvalho and Cordiero show that for any $n\geq 2$ there exists a homeomorphism with shadowing that is $n$-expansive but not $(n-1)$-expansive. Here we provide an example of a surjective system with shadowing that is positively $n$-expansive but not positively $(n-1)$-expansive on the forward orbits as per the definition above. When $n=2$, our example does not have h-shadowing, meaning thst this serves as a counterexample to the would-be natural generalisition of (1) in Corollary \ref{corollaryUniqueshad_hshad_limitshad_s-lim_shad}; that is, $n$-shadowing is not necessarily equivalent to $h$-shadowing and positive $n$-expansivity . We note, however, that the system does have two-sided s-limit shadowing. 


\begin{example}\label{ExampleNexpansive}
Fix $n\geq 2$. Firstly, we define a subset $X_0$ of $\mathbb{R}^2$ recursively in the following manner. Let $Y_0=\{(3,0)\}$. Given sets $Y_0,\ldots, Y_k$, one obtains $Y_{k+1}$ by considering the point $(x,0)\in Y_k$ with the smallest first coordinate. Let $Y_{k+1}$ consist of the points $(x-2^{-k},0)$ and $(x-2^{-k}-2^{-(k+1)},0)$ along with $n-2$ points on the straight line segment whose endpoints are $(x-2^{-k},0)$ and $(x-2^{-k}-2^{-(k+1)},0)$ such that all $n$ of the points are equidistant. Thus, each $Y_k$ for $k\geq 1$ contains exactly $n$ points that are equally spaced along the $x$-axis. Moreover, by construction, all points in each $Y_k$ have positive first coordinate and are distinct and in addition, $Y_i\cap Y_j=\emptyset$ for all $i\neq j$. Let $X_0=\overline{\bigcup_{k=0}^\infty Y_k}=\bigcup_{k=0}^\infty Y_k\cup\{(0,0)\}$. One then defines the sets $X_i$ recursively. Given $X_0,\ldots,X_k$, let $(p,q)\in X_k$ be the point such that $p$ is maximal over the collection of all first coordinates of points in $X_k$ (the second coordinate will be the same for all points in $X_k$ by construction). Then, define $X_{k+1}$ as
	\begin{align*}
    	X_{k+1}=\{(x,y)\in \mathbb{R}^2\mid (x,y+2^{-k})\in X_k\setminus\{(p,q)\}\}.
    \end{align*}
One then defines $X=\bigcup_{k=0}^\infty X_k\cup\{(0,-2)\}$ so that $X$ is closed in $\mathbb{R}^2$. One may then endow $X$ with the standard metric from $\mathbb{R}^2$ to form a metric space. \par 
Next, one defines a map $f:X\to X$ in the following manner. Let $(0,0)$, $(3,0)$ and $(0,-2)$ each be fixed points. For each $k\geq 1$, let $f$ map each point in $Y_k$ to the point in $Y_{k-1}$ with minimal first coordinate so that this defines $f$ on the entirety of $X_0$. For each $k\geq 1$, define $f$ on $X_k$ to map the point $(x,y)\in X_k\setminus\{(0,-\sum_{i=0}^{k-1}2^{-i})\}$ to the point $(z,y+2^{-(k-1)})\in X_{k-1}$ where 
	\begin{align*}
    	z=\min\{x'\mid x<x'\ \text{and}\ (x',y+2^{-(k-1)})\in X_{k-1}\},
    \end{align*}
and then let $f:(0,-\sum_{i=0}^{k-1}2^{-i})\mapsto (0,-\sum_{i=0}^{k-2} 2^{-i})$ for $k\geq 2$, and $f:(0,-1)\mapsto (0,0)$ for $k=1$. By construction, $f$ is then a continuous surjection.\par
Moreover, it is positively $n$-expansive. Indeed, take $r=1/4$ and suppose firstly that $x\in X$ is not one of the fixed points. By construction, the set $\Gamma_+(x,r)$ can contain no points from an $X_k$ different to that containing $x$. Indeed, if $x\in X_0$ this is clear since points in $X_k$ for $k\geq 1$ are at least a distance of 1 away from $x$. Moreover, if $x\in X_k$ for some $k\geq 1$ then there is an iterate of $x$ that is in $X_1$ and so the corresponding iterate of any point that began on $X_j$ for some $j\neq k$ will not be on $X_1$ and hence must be at least a distance of $1/2$ from the iterate of $x$. So, consider firstly the case when $x\in X_0$. Suppose that $j\geq 1$ is such that $x\in Y_j$ then by construction, no point in $X_0$ that is not in $Y_j$ can be in $\Gamma_+(x,r)$ since there will be an iterate of $x$ that is in $Y_1$, and the corresponding iterate of the points not in $Y_j$ will not be in $Y_1$ and hence will be at least a distance of $1/2$ away. Thus, $\Gamma_+(x,r)\subseteq Y_j$ and so since $Y_j$ consists of $n$ points, $|\Gamma_+(x,r)|\leq n$ . Suppose now then that $x\in X_k$ for some $k\geq 1$ then by construction, there exists an $\ell\in\mathbb{N}_0$ such that all points in $\Gamma_+(x,r)$ map onto $X_0$ for the first time under the $\ell$th iterate. Since $f$ is injective on the points in $X\setminus X_0$, each of these $\ell$th iterates must be distinct in $X_0$ and so from the case described previously where $x$ originated in $X_0$, this means that there can be at most $n$ points in $\Gamma_+(x,r)$. It remains to check the fixed points. If $x=(3,0)$, then there is no other point that lies within a distance of $1/4$ from it so $\Gamma_+(x,r)=\{x\}$. If $x=(0,0)$ or $(0,-2)$, then every point that lies with a distance of $1/4$ from it has some iterate that is equal to $(3,0)$ and hence has distance greater than $1/4$ from it so that in these cases also, $\Gamma_+(x,r)=\{x\}$. Thus, $(X,f)$ is positively $n$-expansive. \par 
Conversely, $(X,f)$ is not $(n-1)$-expansive. Indeed, suppose there were such an $r>0$ that exhibited this type of expansivity. Select $k>0$ such that $2^{-k}<r$. Let $x\in Y_k$, then note that by construction each point in $Y_k$ has distance less than $r$ from $x$ and has the same image under $f$. Thus $Y_k\subseteq\Gamma_+(x,r)$, so that $n\leq\Gamma_+(x,r)$ (in fact it is equal by $n$-expansivity). Hence, $(X,f)$ is not $(n-1)$ expansive.
\end{example}
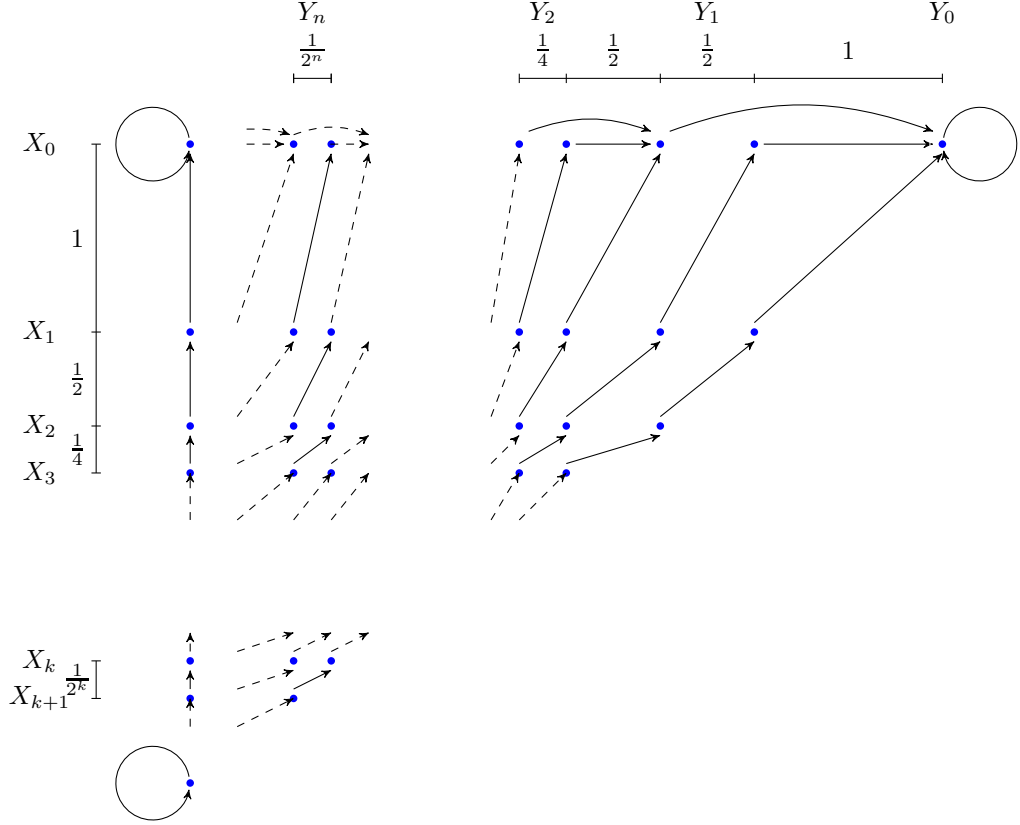
\begin{figure}[h!]
\centering
\begin{tikzpicture}[>=stealth',scale=2.5]
	\draw (3,.35) -- (0.75,.35);
	\foreach \x in {3,2,1.5,1,0.75,-.45,-.25,-1}
		{
		\fill[blue] (\x,0) circle (.2mm);
		}
	\foreach \x in {3,2,1.5,1,0.75,-.45,-.25}
		{	
		\draw (\x,.325) -- (\x,.375);
		}
	\foreach \y/\z in {2.5/1,1.75/\frac{1}{2},1.25/\frac{1}{2},0.875/\frac{1}{4}}
		{
		\node at (\y,.5) {$\z$};
		}	
	\foreach \y/\z in {3/0,1.75/1,0.875/2}
		{
		\node at (\y,.7) {$Y_\z$};
		}	
	\draw (-0.45,.325) -- (-0.45,.375);
	\draw (-0.25,.325) -- (-0.25,.375);
	\draw (-0.45,.35) -- (-0.25,.35);
	\node at (-.35,.5) {$\frac{1}{2^n}$};
    \node at (-.35,.7) {$Y_n$};
	
	\foreach \x in {2,1.5,1,0.75,-.45,-.25,-1}
		{
		\fill[blue] (\x,-1) circle (.2mm);
		}
	\draw (-1.525,0) -- (-1.475,0);
	\draw (-1.525,-1) -- (-1.475,-1);
	\draw (-1.5,0) -- (-1.5,-1);
	\node at (-1.6,-.5) {$1$};
    \node at (-1.8,0) {$X_0$};
	
	\foreach \x in {1.5,1,0.75,-.45,-.25,-1}
		{
		\fill[blue] (\x,-1.5) circle (.2mm);
		}
	\draw (-1.525,-1.5) -- (-1.475,-1.5);
	\draw (-1.5,-1) -- (-1.5,-1.5);
	\node at (-1.6,-1.25) {$\frac{1}{2}$};
    \node at (-1.8,-1) {$X_1$};
	
	\foreach \x in {1,0.75,-.45,-.25,-1}
		{
		\fill[blue] (\x,-1.75) circle (.2mm);
		}
	\draw (-1.525,-1.75) -- (-1.475,-1.75);
	\draw (-1.5,-1.5) -- (-1.5,-1.75);
	\node at (-1.6,-1.625) {$\frac{1}{4}$};
    \node at (-1.8,-1.5) {$X_2$};
	
	\foreach \x in {-.45,-.25,-1}
		{
		\fill[blue] (\x,-2.75) circle (.2mm);
		}
	\draw (-1.525,-2.75) -- (-1.475,-2.75);
	
	\foreach \x in {-.45,-1}
		{
		\fill[blue] (\x,-2.95) circle (.2mm);
		}
	\draw (-1.525,-2.95) -- (-1.475,-2.95);
	\draw (-1.5,-2.75) -- (-1.5,-2.95);
	\node at (-1.6,-2.85) {$\frac{1}{2^k}$};
	\node at (-1.8,-1.75) {$X_3$};
    \node at (-1.8,-2.75) {$X_k$};
    \node at (-1.8,-2.95) {$X_{k+1}$};
    
	\foreach \x/\y in {2.05/2.95,1.05/1.45}
		{
		\draw[->] (\x,0) -- (\y,0);
		}
	\foreach \x/\y in {0.75/1,1.5/2,3/3}
		{
		\node (\y) at (\x,0.05) {};
		}
	\foreach \x/\y in {1/2,2/3}
		{
		\draw[->] (\x) to[out=20,in=160] (\y);
		}
	\draw[dashed,->] (-0.45,0.05) to[out=20,in=160] (-0.05,0.05);
	\draw[dashed,->] (-0.255,0) -- (-0.05,0);
	\node[circle,draw=white] (origin) at (-1,0) {};
	\drawloop[->,stretch=0.5]{origin}{100}{260};
	\node[circle,draw=white] (end) at (3,0) {};
	\drawloop[<-,stretch=0.5]{end}{280}{440};
	\draw[dashed,->] (-.7,0) -- (-.5,0);
	\draw[dashed,->] (-.7,.08) to[out=0,in=160] (-0.47,0.05);
	
	\foreach \y/\x in {3/2,2/1.5,1.5/1,1/0.75,-.25/-.45,-1/-1}
		{
		\draw[->] (\x,-.95) -- (\y,-.05);
		}
	\draw[dashed, ->] (.6,-.95) -- (.75,-.05);
	\draw[->,dashed] (-.25,-.95) -- (-.05,-0.05);
	\draw[->,dashed] (-.75,-.95) -- (-.45,-0.05);
	
	\foreach \y/\x in {2/1.5,1.5/1,1/0.75,-.25/-.45,-1/-1}
		{
		\draw[->] (\x,-1.45) -- (\y,-1.05);
		}
	\draw[dashed, ->] (.6,-1.45) -- (.75,-1.05);
	\draw[->,dashed] (-.25,-1.45) -- (-.05,-1.05);
	\draw[->,dashed] (-.75,-1.45) -- (-.45,-1.05);
		
	\foreach \y/\x in {1.5/1,1/0.75,-.25/-.45,-1/-1}
		{
		\draw[->] (\x,-1.7) -- (\y,-1.55);
		}
	\draw[dashed, ->] (.6,-1.7) -- (.75,-1.55);
	\draw[->,dashed] (-.25,-1.7) -- (-.05,-1.55);
	\draw[->,dashed] (-.75,-1.7) -- (-.45,-1.55);
		
	\draw[dashed,->] (-.45,-2.7) -- (-.25,-2.6);
	\draw[dashed,->] (-.25,-2.7) -- (-.05,-2.6);
	\draw[->] (-.45,-2.9) -- (-.25,-2.8);
	\draw[dashed, ->] (-.75,-2.9) -- (-.45,-2.8);
	\draw[->] (-1,-2.9) -- (-1,-2.8);
	\draw[dashed,->] (-1,-2.7) -- (-1,-2.6);
	\draw[dashed, ->] (-1,-3.1) -- (-1,-2.95);
	\draw[dashed,->] (-.75,-3.1) -- (-.45,-2.95);	
	\draw[dashed,->] (-.75,-2.7) -- (-.45,-2.6);	
	
	\foreach \y/\x in {1/0.75,-.25/-.45,-1/-1}
		{
		\draw[dashed, ->] (\x,-2) -- (\y,-1.75);
		}
	\draw[dashed, ->] (.6,-2) -- (.75,-1.75);
	\draw[->,dashed] (-.25,-2) -- (-.05,-1.75);
	\draw[->,dashed] (-.75,-2) -- (-.45,-1.75);
	
	\node[circle,draw=white] (origin) at (-1,-3.4) {};
	\drawloop[->,stretch=0.5]{origin}{100}{260};
	\fill[blue] (-1,-3.4) circle (.2mm);
\end{tikzpicture}
\caption{The construction from Example \ref{ExampleNexpansive} for a $2$-expansive map}
\end{figure}

\begin{remark}
It is known that if a system on a compact space has shadowing and is expansive then it is topologically stable (see \cite{WaltersP}). By Theorem \ref{thmnShadIFF} it can be equivalently said that compact systems with unique shadowing are topologically stable.
\end{remark}

Clearly there is more to be said on uniqueness and how it modifies other forms of shadowing. For example, one may define suitably `unique' variants of the other shadowing types mentioned in this paper, i.e.\ s-limit shadowing, h-shadowing and limit shadowing.

\begin{definition}Let $(X,f)$ be a dynamical system. The map $f$ has \emph{unique s-limit shadowing} if
\begin{enumerate}
\item it has unique shadowing; and,
\item there exists $\eta>0$ such that for any $\epsilon>0$ with $\epsilon< \eta$ there exists $\delta>0$ such that every asymptotic $\delta$-pseudo-orbit is asymptotically $\epsilon$-shadowed by a unique point.
\end{enumerate}
\end{definition}

We note that postulating uniqueness in condition (2) is of course unnecessary in virtue of condition (1).  On the other hand, unlike the situation for for s-limit shadowing, it is not clear that condition (2) implies condition (1).

\begin{definition}
Let $(X,f)$ be a dynamical system. The map $f$ has \emph{unique h-shadowing} if there exists $\eta>0$ such that for any $\epsilon>0$ there exists $\delta>0$ such that for any finite $\delta$-pseudo-orbit $(x_0, \ldots , x_n)$ there exists a unique point $z$ such that $d(f^i(z), x_i) < \epsilon$ for each $i \in \{0, \ldots, n\}$ and $f^n(z)=x_n$.
\end{definition}

\begin{definition} Let $(X,f)$ be a dynamical system. The map $f$ has \emph{unique limit shadowing} if every asymptotic pseudo-orbit is asymptotically shadowed by a unique point.
\end{definition}

The proofs of Theorems  \ref{corollaryUnique_s_limShad}, \ref{corollaryUniqueShadImpliesUnique_h-shad} and \ref{corollaryUniqueLimitShad} come easily given our previous discussions and are thereby omitted. It is worth remarking upon how `uniqueness' modifies the various properties: for example, on a compact space shadowing is strictly weaker than h-shadowing \cite{BarwellGoodOprocha} but Theorem \ref{corollaryUniqueShadImpliesUnique_h-shad} and Example \ref{example_Unique_h_shad_not_unique_shad} together entail that unique shadowing is strictly stronger than unique h-shadowing.

\begin{theorem}
\label{corollaryUnique_s_limShad}
Let $(X,f)$ be a dynamical system. Then the map $f$ has
\begin{enumerate}
    \item unique shadowing if and only if it has unique s-limit shadowing.
    \item two-sided unique shadowing if and only if it has two-sided unique s-limit shadowing.
\end{enumerate}
\end{theorem}

\begin{theorem}\label{corollaryUniqueShadImpliesUnique_h-shad} Let $(X,f)$ be a dynamical system. If $f$ has unique shadowing then it has unique h-shadowing.
\end{theorem}

Example \ref{example_Unique_h_shad_not_unique_shad} shows that the converse to Theorem \ref{corollaryUniqueShadImpliesUnique_h-shad} is false.

\begin{example}\label{example_Unique_h_shad_not_unique_shad} Consider $X= \{1/2^n \mid n \in \mathbb{N}_0\} \cup \{0\}$ and let $f$ be the identity map on $X$. Then $f$ has unique h-shadowing but, by Corollary \ref{corollaryUniqueshad_hshad_limitshad_s-lim_shad}, not unique shadowing because it is not positively expansive.
\end{example}

\begin{theorem}\label{corollaryUniqueLimitShad} Let $(X,f)$ be a dynamical system. The map $f$ has unique limit shadowing if and only if it has limit shadowing and no asymptotic pairs.  Moreover, if $f$ has unique shadowing, then $f$ is injective. 
\end{theorem}

As with classical shadowing and s-limit shadowing, limit shadowing has a two-sided analogue: A system $(X,f)$ has \emph{(unique) two-sided limit shadowing} if for any two-sided asymptotic pseudo-orbit $(x_i)_{i \in \mathbb{Z}}$ there exists a (unique) two-sided orbit $(z_i)_{i \in \mathbb{Z}}$ which asymptotically shadows it (i.e.\ $d(z_i, x_i) \to 0$ as $i \to \pm\infty$). Two-sided limit shadowing has recently attracted an array of interest (e.g.\ \cite{Carvalho2015, Carvalho2018, CarvalhoKwietniak, Oprocha2014}). Of particular note, is its strength as a condition: it is among the strongest of the pseudo-orbit tracing properties. For homeomorphisms, it has been shown to imply shadowing, mixing and the specification property \cite{CarvalhoKwietniak}. We close this paper by examining how uniqueness modifies two-sided limit shadowing. Since our map is not necessarily a homeomorphism, we first require some additional terminology. Given a continuous self-map $f \colon X \to X$ on a compact metric space  $X$, the set $K_f=\bigcap_{n \in \mathbb{N}} f^n(X)$, which might be termed the {surjective core} of $f$, is a nonempty set on which $f$ is surjective (see, for example \cite{good-greenwood-et-al}). We may then define the \emph{induced core system} $(K_f, f\restriction_{K_f})$, which is easily seen to be a surjective dynamical system. We omit the proof of the following lemma.

\begin{lemma}\label{Lemma_Induced_core_twosidedlimit}
If $(X,f)$ has two-sided limit shadowing then the induced core system has two-sided limit shadowing.
\end{lemma}

Recall that a system $(X,f)$ is \emph{transitive} if for any pair of nonempty open sets $U$ and $V$ there exists $n \in \mathbb{N}$ such that $f^n(U) \cap V \neq \emptyset$. It is \emph{mixing} if for any such pair there exists $N \in \mathbb{N}$ such that $f^n(U) \cap V \neq \emptyset$ for all $n \geq N$. It is well-known that if $f$ is a transitive surjection then the system $(X,f)$ either consists of a single periodic orbit or $X$ contains at least continuum many points (with none being isolated). In similar fashion, it is easily observed that if $f$ is a mixing surjection then the system $(X,f)$ either consists of a single fixed point or $X$ contains at least continuum many points (with none being isolated).

\begin{theorem}\label{Theorem_Shad_mixing_injective_notpositivelynexpansive}
 Let $(X,f)$ be a dynamical system, where $f$ is an injective map with two-sided shadowing. If $f$ is mixing and $X$ contains more than one point then it is not positively $n$-expansive for any $n$.
\end{theorem}

\begin{proof}
Note first that a transitive system on a compact space is onto, so $f$ is a homeomorphism. Let $n >1$ be given. Since $f$ is mixing and $X$ consists of more than one point, $X$ is infinite. Let $A=\{x_0, \ldots , x_n, z\}$ be a set of $n+2$ distinct points in $X$. Let $\epsilon>0$ be such that the $\epsilon$-balls around points in $A$ are pairwise disjoint and let $\delta>0$ satisfy the shadowing condition for $\epsilon/2$. Without loss of generality $\delta < \epsilon/2$. By mixing, there exists $n_{-1} >1$ and, for each $i \in \{0, \ldots, n\}$, $x_i ^{(-1)} \in B_{ \frac{\delta}{2}}\left(x_i \right)$ such that 
\[ f^{n_{-1}} \left( x_i ^{(-1)} \right) \in B_{ \frac{\delta}{2}}\left(x_{i+1 \pmod{n+1}} \right).\]
Recursively by mixing there exists, for each $j<-1$, $n_{j} >1$ and there exists, for each $i \in \{0, \ldots, n\}$, $x_i ^{(j)} \in B_{ {2^{j}\delta}}\left(x_i \right)$ such that
\[ f^{n_j} \left( x_i ^{(j)} \right) \in B_{ 2^{j}\delta}\left(x_{i+1 \pmod{n+1}} \right).\]
Finally, for each $i \in \{0, \ldots, n\}$ there exists $x^\prime _i \in B_\frac{\delta}{2}(x_i)$ and $m \in \mathbb{N}$ such that $f^m(x^\prime _i) \in B_\delta(z)$. For each $i \in \{0, \ldots, n\}$ we now have a two-sided (asymptotic) $\delta$-pseudo orbit:
\begin{align*}
&\Bigg(  \ldots, x^{(j)} _{i+j \pmod{n+1}}, f\left(x^{(j)} _{i+j \pmod{n+1}}\right), f^2\left(x^{(j)} _{i+j \pmod{n+1}}\right), \ldots,
\\&f^{n_j -1}\left(x^{(j)} _{i+j \pmod{n+1}}\right), x^{(j+1)}_{i+j+1 \pmod{n+1}}, f\left( x^{(j+1)}_{i+j+1 \pmod{n+1}}\right),\ldots , 
\\ &f^{n_{j+1}-1}\left( x^{(j+1)}_{i+j+1 \pmod{n+1}}\right), x^{(j+2)}_{i+j+2 \pmod{n+1}}, \ldots, 
\\&\ldots, x^{(-1)} _{i-1 \pmod{n+1}}, f\left(x^{(-1)} _{i-1 \pmod{n+1}}\right), \ldots, f^{n_{-1}-1}\left(x^{(-1)} _{i-1 \pmod{n+1}}\right),
\\ &x^\prime _i, f\left(x^\prime _i\right), f^2\left(x^\prime _i\right), \ldots, f^{m-1}\left(x^\prime _i\right), z, f(z), f^2(z),\ldots \Bigg)
\end{align*}
where the $0$\textsuperscript{th} term is given by $x^\prime _i$ for each such $i$. Note that these pseudo-orbits are distinct. By two-sided shadowing these pseudo-orbits are $\epsilon/2$-shadowed. Notice that each one is shadowed by a distinct point since, for each distinct pair $i,j \in \{0, \ldots, n\}$, 
\begin{align*}
    d(x_i ^\prime, x_j ^\prime) &>d(x_i, x_j) - \delta
    \\ &> d(x_i, x_j) - \epsilon/2
    \\ &> 3\epsilon/2.
    \end{align*}
Let $y_i \in X$ be a point which $\epsilon / 2$-shadows the pseudo-orbit through $x^\prime _i$ (so that $y_i \in B_\frac{\epsilon}{2}(x^\prime _i)$). Since $f$ is injective $f^k(y_i) \neq f^k(y_j)$ for any $k \in \mathbb{N}$ and distinct $i$ and $j$. It remains now to observe that for each $k \geq m$ and all $i \in \{0, \ldots, n\}$ we have $f^k(y_i) \in B_\frac{\epsilon}{2}(f^{k-m}(z))$. In particular $f^m(y_i) \in \Gamma_+(z, \epsilon)$ for each $i \in \{0, \ldots n\}$ and so $\lvert \Gamma_+(z, \epsilon)\rvert \geq n+1$. Since we could have chosen $\epsilon$ arbitrarily small it follows that $(X,f)$ is not positively $n$-expansive for any $n$. 
\end{proof}


\begin{corollary}
 Let $(X,f)$ be a dynamical system, where $f$ is an injective map with two-sided limit shadowing. If the surjective core contains more than one point then $f$ is not positively $n$-expansive for any $n$.
\end{corollary}

\begin{proof}
By Lemma \ref{Lemma_Induced_core_twosidedlimit} the induced core system has two-sided limit shadowing and therefore, by \cite[Theorem B]{CarvalhoKwietniak}, is mixing. By \cite[Theorem A]{CarvalhoKwietniak} the induced core system has two-sided shadowing. Therefore, by Theorem \ref{Theorem_Shad_mixing_injective_notpositivelynexpansive} the induced core system is not positively $n$-expansive for any $n \in \mathbb{N}$. It immediately follows that neither is $(X, f)$.
\end{proof}

\begin{corollary}
An injective map with unique two-sided limit shadowing does not have $n$-shadowing for any $n\in \mathbb{N}$. In particular, it does not have unique shadowing. 
\end{corollary}

\bibliographystyle{plain} 
\bibliography{bib}

\end{document}